\documentclass[11pt]{amsart}
\usepackage{graphicx}
\usepackage{mathpazo}
\usepackage{amssymb}
\usepackage{epstopdf}
\usepackage[cmtip,all]{xy}
\usepackage{enumerate}
\usepackage{stmaryrd}
\usepackage{hyperref}
\newtheorem{prop}{Proposition}
\newtheorem{lemma}[prop]{Lemma}

\usepackage{amsmath}
\usepackage{mathrsfs}
\usepackage{amsthm}
\usepackage{amssymb}
\usepackage{amsfonts}
\usepackage{tikz} 
\usepackage{mathtools}
\usepackage{setspace}
\usepackage{cancel}
\usepackage{chemarrow}
\usetikzlibrary{matrix,arrows} 
\usepackage{amscd}

\newtheorem{theorem}{Theorem}
\newtheorem*{theorem*}{Theorem}

\newtheorem*{cor*}{Corollary}
\newtheorem*{prop*}{Proposition}
\newtheorem{cor}[prop]{Corollary}

\theoremstyle{definition}

\newtheorem{remark}[prop]{Remark}

\newcommand{\CC}{\mathbb{C}}
\renewcommand{\P}{\mathbb{P}}

\newcommand{\C}{\mathfrak{C}}

\newcommand{\I}{\mathcal{I}}

\renewcommand{\O}{\mathcal{O}}

\newcommand{\Q}{\mathbb{Q}}
\newcommand{\F}{\mathcal{F}}

\renewcommand{\c}{\mathsf{c}}

\newcommand{\U}{\mathcal{U}}
\renewcommand{\S}[1]{S^{[#1]}}

\newcommand{\SC}[1]{(S/C)^{[#1]}}

\newcommand{\PCoo}[1]{(\P_{C,L}/C_\infty)^{[#1]}}
\newcommand{\PR}[1]{(\P_{C,L})^{[#1],\sim}}
\newcommand{\Po}[1]{(\P_{C,L,0})^{[#1]}}
\newcommand{\T}[1]{\mathcal{T}^{[#1]}}
\newcommand{\Pic}{\text{Pic}}

\title[Intersection numbers on the relative Hilbert schemes]{Intersection numbers on the relative Hilbert schemes of points on surfaces}
\author{Amin Gholampour, Artan Sheshmani}
\date{\today}                                           

\begin{document}
\maketitle
\begin{abstract}
We study certain top intersection products on the Hilbert scheme of points on a nonsingular surface relative to an effective smooth divisor. We find a formula relating these numbers to the corresponding intersection numbers on the non-relative Hilbert schemes. In particular, we obtain a relative version of the explicit formula found by Carlsson-Okounkov for the Euler class of the twisted tangent bundle of the Hilbert schemes. 
\end{abstract}
\section{Introduction}
\subsection{Overview}
Hilbert schemes have long been regarded as one of the important objects of study in algebraic geometry. In particular, in the case of the Hilbert scheme of points on surfaces, there have been many exciting achievements, both in mathematics and physics (see \cite{N99} for a survey). One of the interesting aspects of the study of these objects is their enumerative geometry. Our main motivation for studying enumerative problems arising over the Hilbert scheme of points on surfaces is their relation to the Donaldson-Thomas invariants of 2-dimensional sheaves inside threefolds.

Let $S$ be a nonsingular projective surface over $\CC$. In 1990, G\"{o}ttsche \cite{G90} computed the generating series associated to the integral of the Euler class of the tangent bundle (Euler characteristic) of the Hilbert scheme of $n$ points, denoted by $S^{[n]}$, and the result appeared to be closely related to the Dedekind eta modular form:

\begin{theorem*}(G\"{o}ttsche)
$$1+ \sum_{n>0 }q^n \int_{S^{[n]} }e(\S{n})=\Xi(q)^{-e(S)}$$
where $$\Xi(q)=\prod_{n>0}(1-q^n).$$
\end{theorem*}

More recently,  Carlsson and Okounkov  \cite{CO12} found an analog of G\"{o}tt\-sche's formula for the integral of the Euler class of the tangent bundles of $\S{n}$ \emph{twisted} by a line bundle on $S$ as we will explain below. These Euler classes arise naturally in computing the Donaldson-Thomas invariants of 2-dimensional sheaves in threefolds containing $S$ as a divisor.

 Given $S$ and a line bundle $M$ on $S$, the twisted tangent bundle $\T{n}_S(M)$ is a bundle defined on $\S{n}$, whose fiber over the ideal sheaf $I \in \S{n}$ is naturally identified with \begin{equation} \label{fiber} \sum_{i=0}^2(-1)^i(H^i(M)-Ext^i(I,I\otimes M)).\end{equation} We will give a precise definition of $\T{n}_S(M)$ in Section \ref{application}.

\begin{theorem*}  (Carlsson-Okounkov) \begin{equation} \label{eqco} 1+ \sum_{n>0 }q^n \int_{S^{[n]} }e(\T{n}_S(M))=\Xi(q)^{-\delta_S(M)},\end{equation} where $$\delta_S(M)=\int_S e(TS\otimes M).$$ \qed\end{theorem*} 
 Note that when $M=\O_S$ the theorem above specializes to G\"{o}ttsche's formula. One of the key facts used in this paper is the following feature of the theorem above:

\begin{remark} \label{quasi} If $S$ is quasi-projective equipped with a $\CC^*$-action so that $S^{\CC^*}$ is complete, and $M\in \Pic(S)^{\CC^*}$ (the equivariant Picard group of $S$), then the integrals in the theorem can be defined by the equivariant residues, and in this sense, the formula in the theorem still holds \cite {CO12}.
\end{remark}

Among the other important top intersection products on $\S{n}$ is the integral of the top Segre class of the tautological bundles associated to a line bundle on $S$. These integrals were studied in \cite{L99}.
\subsection{Main results} \label{sec:main}

The main object of  study in this paper is the relative Hilbert scheme of points. Let $C$ be a smooth effective divisor on $S$. Li and Wu \cite{LW15} introduced the notion of a \emph{stable relative} ideal sheaf.
$I\in \S{n}$ is said to be relative to $C$ if the natural map $$I\otimes \O_C\to \O_S\otimes \O_C$$ is injective (see also \cite{MNOP06}). Relativity is an open condition on $\S{n}$. Li and Wu constructed a proper relative Hilbert scheme, denoted by $\SC{n}$, by considering the equivalence classes of the stable relative ideal sheaves on the $k$-step semistable models $S[k]$ for $0\le k\le n$. Two relative ideal sheaves $I_1$ and $I_2$ on $S[k]$ are equivalent if the quotients $\O_{S[k]}/I_i$ differ by an automorphism of $S[k]$ covering the identity on $S[0]=S$. The stability of a relative ideal sheaf means that it has finitely many auto equivalences as described above.  $\SC{n}$ is a smooth proper Deligne-Mumford stack of dimension $2n$. 

We study some specific top intersection products over $\SC{n}$.
 We express the intersection numbers over  $\SC{n}$ in terms of the intersection numbers over the standard (non-relative) Hilbert scheme of points. 
 
 One of the key ingredients of the proof of the main result in this paper is the degeneration technique. Let $S  \rightsquigarrow S_0:= S_1\cup_C S_2$ be a good degeneration of the surface $S$, and let \begin{equation}\label{expanded}\mathfrak{S}\to \C\end{equation} be the corresponding universal family over the \emph{stack of expanded degenerations} (see \cite{L01, L02, LW15}).   Li and Wu [LW15] constructed the Hilbert scheme of $n$ points  on this family that we denote by $$\mathfrak{H}^{[n]}(\mathfrak{S}\to \C).$$ The Hilbert scheme $\mathfrak{H}^{[n]}(\mathfrak{S}\to \C)$ is a proper Deigne-Mumford stack over the 1-dimensional Artin stack $\C$ parameterizing the expanded degenerations \eqref{expanded}. A non-special fiber of $\mathfrak{H}^{[n]}(\mathfrak{S}\to \C)$ over $\C$ is isomorphic to $\S{n}$, whereas the special fiber of $\mathfrak{H}^{[n]}(\mathfrak{S}\to \C)$, denoted by  $S^{[n]}_0$, can be written as the (non-disjoint) union \begin{equation} \label{central}S^{[n]}_0=\bigcup_{\tiny \begin{array}{c}n_1, n_2 \ge 0\\ n=n_1+n_2 \end{array}} (S_1/C)^{[n_1]}\times (S_2/C)^{[n_2]}. \end{equation}

Our main theorem (Theorem \ref{main1}) is stated and proven in a general setting that can be applied in various situations. As applications of Theorem \ref{main1}, we will deduce the extensions of Carlsson-Okounkov's formula and some other known results to the set up of the relative Hilbert schemes. For this, we work with an assignment $A$ that takes a pair $(S,C)$ as above, a line bundle on $S$, and a nonnegative integer $n$. The output of the  assignment $A$ is an element of the $K$-group of the relative Hilbert scheme $\SC{n}$. The assignment $A$ is required to be \emph{well-behaved} under the good degenerations of $S$ as above, and also to \emph{respect} the $\CC^*$-equivariant structures. More precisely,  for any $n=0, 1, 2,\dots,$ suppose that  $$A_{S/C}^{[n]}:\Pic(S) \to K(\SC{n})$$ is an assignment satisfying the following properties (if $C=0$ we drop ``$/C$'' from the notation): 
\begin{enumerate}[(i)]
\item If $n=0$ then we take $\SC{0}=\text{Spec }\CC$, and $A_{S/C}^{[0]}(M)=1$.
\item Given a good degeneration $S  \rightsquigarrow S_1\cup_C S_2$, then there exist the degenerations of the line bundle $M$ and the K-group element $A_S^{[n]}(M)$ $$  \Pic(S)\ni M \rightsquigarrow  M_i \in \Pic(S_i) \quad i=1, 2,$$ $$A_S^{[n]}(M)\rightsquigarrow \mathcal{K} \in K(S^{[n]}_0),$$  such that for any $$n_1, n_2\ge 0 \quad \text{and} \quad n=n_1+n_2$$ the restriction of  $\mathcal{K}$ to the corresponding component of the special fiber \eqref{central} of $\mathfrak{H}^{[n]}(\mathfrak{S}\to \C)$ is given by $$ \mathcal{K}|_{(S_1/C)^{[n_1]}\times (S_2/C)^{[n_2]}} \cong A_{S_1/C}^{[n_1]}(M_1)\boxplus A_{S_2/C}^{[n_2]}(M_2).$$ 
\item If $S$ is equipped with a $\CC^*$-action that leaves $C$ invariant, and $M$ admits a $\CC^*$-equivariant lift $M^*\in \Pic(S)^{\CC^*}$, then $$A_{S/C}^{[n]}(M^*)\in K(\SC{n})^{\CC^*}.$$ Here $\Pic(-)^{\CC^*}$ and $K(-)^{\CC^*}$ indicate the equivariant Picard and $K$-groups.
\end{enumerate}

For the virtual bundle $u-v$,  the Euler class $e(u-v)$ is defined to be the top degree graded piece of $c(u)/c(v)$, where $c(-)$ is the total Chern class. We are interested in studying the following generating series:
\begin{equation}\label{K-gen}
Z_A(S/C,M)=1+ \sum_{n>0 }q^n \int_{\SC{n}}e(A_{S/C}^{[n]}(M)).
\end{equation}

\begin{remark} \label{quasi} Suppose that $S$ is quasi-projective, but it is equipped with a $\CC^*$-action that leaves $C$ invariant, and so that $S^{\CC^*}$ is complete. If $M^{*}$ is any  $\CC^*$-equivariant lift of $M$, and if $A_{S/C}^{[n]}(M^*)$ for $n=1,2,\dots, $ are the corresponding $\CC^*$-equivariant classes, then, $Z_A(S/C,M^*) \in \mathbb{Q}[t, t^{-1}]$ can be defined by means of the equivariant residues. 
\end{remark}
If $C=0$ we drop ``$/C$'' from the notation of the generating series. For example, thinking of the twisted tangent bundle $\T{n}_{S/C}(M)$ with fibers given by \eqref{fiber} (see Section \ref{application} for more details), as an assignment $\Pic(S) \to K(\SC{n})$, then Carlsson-Okounkov's formula can be rewritten as \begin{equation}Z_{\mathcal{T}}(S,M)=\Xi(q)^{-\delta_S(M)}.\end{equation}

The main result of the paper expresses $Z_A(S/C,M)$ in terms of the generating series of the same intersection numbers over the non-relative Hilbert schemes:
\begin{theorem}\label{main1}Let $\pi:\mathcal{N}\to C$ be the total space of the normal bundle of $C$ in $S$ equipped with the natural fiberwise $\CC^*$-action. For any $M\in \Pic(S)$, there exists a lift of this action to $\pi^*(M|_C)$,  giving the equivariant line bundle $M^*\in Pic(\mathcal{N})^{\CC^*}$, such that  $Z_A(\mathcal{N} ,M^*)$ has a non-equivariant limit and 
$$Z_A(S/C,M)=\frac{Z_A(S,M)}{ Z_A(\mathcal{N} ,M^*)}.$$
\end{theorem} 

We prove Theorem \ref{main1} in Section \ref{Theory}. In the course of proof, we obtain a few other formulas for the intersection numbers on the relative Hilbert schemes which are interesting on their own. Most importantly, we prove an explicit formula for the generating series of a specific top intersection product over the \emph{rubber} Hilbert scheme (see Corollary \ref{cor:rubber}). 
 
 The main application of Theorem \ref{main1} is the version of Carlsson-Okounkov's formula for the relative Hilbert schemes:
 \begin{theorem}\label{mains}
Let $S$ be a nonsingular projective surface and $C_1,\dots C_k \subset S$ be disjoint nonsingular curves. If $M$ is a line bundle on $S$ then
$$Z_{\mathcal{T}}(S/C_1\cup \cdots \cup C_k,M)=\Xi(q)^{-\int_S e(T_S[-C_1-\dots -C_k]\otimes M)},$$ where $T_S[-C_1-\dots -C_k]$ is the sheaf of tangent fields with logarithmic zeros along the divisor $\sum C_i$. 
\end{theorem} This theorem is proven in Section \ref{application}. Theorem \ref{mains}  is the 2-dimensional analog of the degree 0 relative MNOP conjecture \cite[Conjecture 1R]{MNOP06} proven in \cite{LP09}. 

In Section \ref{other}, we discuss further direct applications of Theorem \ref{main1}. We consider two more important top intersection products on the relative Hilbert schemes. One is the integral of the top Segre class of the tautological bundles, and the other is the self intersection number $D^{[n]}\cdot D^{[n]}$ for the nonsingular curve $D \subset S$. We prove that these numbers can be expressed in terms of the intrinsic invariants of the pair $(S,C)$ (Theorem \ref{thm:intrinsic}). 
One of our motivations for this work is the study of the Donaldson-Thomas invariants of 2-dimensional sheaves inside threefolds. S-duality predicts that these invariants have modular properties \cite{VW94, DM11, GST14}. As an illustration, in Section \ref{DT}, we give an example of the most basic case i.e. sheaves supported on hyperplane sections of $\P^3$. We use Theorem \ref{mains} to give a formula for the generating functions of these DT invariants and establish their modularity.
\section*{Aknowledgement}
We thank Richard Thomas for discussions and providing valuable comments. We would  like to thank Harvard University and the Journal of Differential Geometry for supporting our short visit during the JDG conference, where some of the relevant discussions on this project took place.

A. G. was partially supported by NSF grant DMS-1406788.

\section{Proof of Theorem \ref{main1}}\label{Theory}

We first compute the left hand side of Theorem \ref{main1} in the case where $S$ is the projectivization of a split rank 2 vector bundle over a curve, and $C$ is the  $\infty$-section. This is done by localization (Proposition \ref{localization}). We then use degeneration to the normal cone to prove the theorem for a general $(S,C)$ (Proposition \ref{relsurf}).  

Let $C$ be a nonsingular genus $g$ curve and $L$ a line bundle on $C$. Define \begin{equation}\label{pcl}\P_{C,L}:=\P(\O\oplus L)\xrightarrow{\pi} C.\end{equation}

We denote by $C_0=\P(\O)$ and $C_\infty=\P(L)$, the $0$- and $\infty$-sections of $\P_{C,L}$ respectively. Given $p\in C$, let $f=[\pi^{-1}(p)]$ denote the numerical class of the fiber.

It is easily seen that $\CC^*$ acts on $\P_{C,L}$ by scaling the fibers. Suppose that the weight of the action of $\CC^*$ on the normal bundles of $C_0$ and $C_\infty$ are respectively $t$ and $-t$. Given $M \in \Pic(\P_{C,L})$, we fix a lift of the $\CC^*$-action to $M$ by requiring that the fibers over $C_\infty$ have the trivial weight. We denote the resulting $\CC^*$-equivariant line bundle by $M^\infty$.

Moreover, in the following arguments, we will be using yet another geometry of Hilbert schemes, that is called the rubber Hilbert scheme. The $\CC^*$-action on $\P_{C,L}$ induces an action on the relative Hilbert scheme $$(\P_{C,L}/C_0\cup C_\infty)^{[n]}.$$ Let $\U$ be the open subset with finite $\CC^*$-stabilizers and no degeneration over $C_\infty$. The rubber Hilbert scheme is then defined to be the Deligne-Mumford quotient stack $$\PR{n}=[\U/\CC^*].$$  Note that the rubber Hilbert scheme above carries no $\CC^*$-action. Let $$K(\PR{n})^{\CC^*}$$ be the $\CC^*$-equivariant $K$-group with respect to the trivial action on the rubber Hilbert scheme.

   Now given any $M^* \in \Pic(\P_{C,L})^{\CC^*}$, by the property (iii) of the assignment $A$, we can define $$A_{\P_{C,L}}^{[n],\sim}(M^*)\in K(\PR{n})^{\CC^*}$$ to be the descent of $A_{\P_{C,L}/C_0\cup C_\infty}^{[n]}(M^*)\in K((\P_{C,L}/C_0\cup C_\infty)^{[n]})^{\CC^*}$. Define the generating series 
   \begin{align}\label{rubber} &Z_A(\P_{C,L},M^*)^{\sim} =1+ \sum_{n>0 }q^{n}\int_{\PR{n}} \frac{e(A_{\P_{C,L}}^{[n],\sim}(M^*))}{-t-\Psi_0} \in \mathbb{Q}\llbracket q\rrbracket [t,t^{-1}],\end{align} 
   where  $\Psi_0$ is the first Chern class of the cotangent line bundle of the rubber Hilbert scheme, corresponding to the relative divisor $C_0$.

   The following Proposition aims at proving an analog of Theorem \ref{main1} for the spacial case $(S,C)=(\P_{C,L},C_\infty)$:

\begin{prop} \label{localization} 
$$Z_A(\P_{C,L}/C_\infty,M)=Z_A(\P_{C,L,0},M^\infty)$$ where $\P_{C,L,0}=\P_{C,L} \setminus C_\infty.$ In particular, the right hand side has a non-equivariant limit.
\end{prop}
\begin{proof}
The fixed locus of the $\CC^*$-action on $\PCoo{n}$ is the disjoint union over all products of the form (see \cite[Section 4.3]{MNOP06})
$$\PR{n_1} \times (\Po{n_2})^{\CC^*}$$ where $n_1+n_2=n$.
The first factor is the rubber Hilbert scheme of points as described above, and the second factor is the $\CC^*$-fixed point set of the Hilbert scheme of points on $\P_{C,L,0}$, which is the total space of the line bundle $L^\vee$ over $C$.

Applying Atiyah-Bott localization formula, we get 
$$Z_A(\P_{C,L}/C_\infty,M)=Z_A(\P_{C,L},M^\infty)^{\sim} \cdot Z_A(\P_{C,L,0},M^\infty).$$
Note that here, the $K$-group class $A_{\P_{C,L}}^{[n],\sim}(M^\infty)$ carries no $\CC^*$-weights by the choice of the lift of the action to $M^\infty$. Now observe that since $$\dim \PR{n}=2n-1,$$ then all the integrals on the right hand side of \eqref{rubber} vanish for  dimension reasons and hence, we get $$Z_A(\P_{C,L},M^\infty)^{\sim}=1,$$ and this finishes the proof.
\end{proof}

Given $M \in \Pic(\P_{C,L})$, we now fix a lift of the $\CC^*$-action to $M$ by requiring that the fibers over $C_0$ have the trivial weight. We denote the resulting $\CC^*$-equivariant line bundle by $M^0$.
\begin{cor} \label{cor:rubber}
$$Z_A(\P_{C,L},M^0)^{\sim}=\frac{Z_A(\P_{C,L,0},M^\infty)}{Z_A(\P_{C,L,0},M^0)}$$
\end{cor}
\begin{proof}
Applying Atiyah-Bott localization formula we get 
$$Z_A(\P_{C,L}/C_\infty,M)=Z_A(\P_{C,L},M^0)^{\sim} \cdot Z_A(\P_{C,L,0},M^0).$$ The result now follows from Proposition \ref{localization}.
\end{proof}
Now we use Proposition \ref{localization} and the degeneration technique to complete the proof of Theorem \ref{main1}: 

\begin{prop} \label{relsurf}
Let $(S,C)$ be a pair of a nonsingular surface and a nonsingular divisor, and let $$\P:=\P(\O_C \oplus \O_C(-C))\xrightarrow{\pi} C.$$ Given $M\in \Pic(S)$, let $M_\P=\pi^*(M|_C)$. Then
$$Z_A(S/C,M)=\frac{Z_A(S,M)}{ Z_A(\P/C_\infty,M_\P)}.$$
\end{prop}

\begin{proof}
Consider the degeneration of $S$ to the normal cone of $C$ in $S$ \begin{equation} \label{degen} S \rightsquigarrow S \cup_{C=C_\infty} \P.\end{equation}  The line bundle $M$ degenerates to $$M \rightsquigarrow M \text{ on } S \quad \text{ and } \quad M_\P  \text{ on } \P.$$  

As mentioned in Section \ref{sec:main}, we can associate to the degeneration \eqref{degen}, the stack of expanded degenerations $\C$ together with the universal family $\mathfrak{S}\to \C$, and the Hilbert scheme $\mathfrak{H}^{[n]}(\mathfrak{S}\to \C)$.
The invariance of the degree of the 0-cycles in the fibers of $\mathfrak{H}^{[n]}(\mathfrak{S}\to \C)$ over $\C$ gives 
$$ \int_{\S{n}}e(A_{S}^{[n]}(M))=\sum_{n_1=0}^n \int_{\SC{n_1}}e(A_{S/C}^{[n_1]}(M))\cdot  \int_{(\P/C_\infty)^{[n-n_1]}}e(A_{\P/C_\infty}^{[n-n_1]}(M_\P)),$$ using the properties (i) and (ii) of the assignment $A$ mentioned in Section \ref{sec:main} \footnote{This is is the 2-dimensional analog of the more complicated degeneration formula for the  Hilbert schemes on threefolds studied in \cite[Section 6]{LW15} and \cite{MNOP06}. The 2-dimensional degeneration formula was also employed in \cite{T12, LT14}.}. Multiplying both sides of this formula by $q^n$ and summing over $n\ge 0$, we get
\begin{align*}&Z_A(S,M)=Z_A(S/C,M)\cdot Z_A(\P/C_\infty,M_\P),
\end{align*} from which the result follows immediately.
\end{proof}

Propositions  \ref{localization} and \ref{relsurf}  give the following result, which finishes the proof of Theorem \ref{main1}:
\begin{cor} \label{cor:formul}
$$Z_A(S/C,M)=\frac{Z_A(S,M)}{ Z_A(\P_{C,L,0} ,M^\infty_\P)}.$$
\end{cor}
\hfill $\Box$

\section{Proof of Theorem \ref{mains}}\label{application}
In this case we specialize to the case
$$A_{S/C}^{[n]}=\T{n}_{S/C}:\Pic(S)\to K(\SC{n})$$ defined for any pair $(S,C)$ of a nonsingular suface and a smooth divisor by 
 \begin{equation*}\T{n}_{S/C}(M)=[Rp_{2*}p_1^*M]-[Rp_{2*}R\mathcal{H}om(\I,\I\otimes p_1^*M)],\end{equation*} where $p_1,p_2$ are projections to the first and second factors of $S\times \SC{n}$, and $\I$ denotes the universal ideal sheaf. 

\begin{lemma}\label{satis}
The assignment $\mathcal{T}$ satisfies properties (i), (ii), (iii) in Section \ref{sec:main}. 
\end{lemma}
\begin{proof}
Properties (i) and (iii) are clearly satisfied by $\mathcal{T}$. To verify property (ii), suppose we are given a good degeneration $S  \rightsquigarrow S_0:=S_1\cup_C S_2$ and the degeneration of the line bundle $M$ on $S$,  $$M\rightsquigarrow M_0 \quad \text{with} \quad M_i:=M_0|_{S_i}\in \text{Pic}(S_i),\quad M_C:=M|_C,$$ together with the decomposition  $n=n_1+n_2$. Let $S_0^{[n]}$ denote the special fiber of $\mathfrak{H}^{[n]}(\mathfrak{S}\to C)$ which can be written as \eqref{central}. Tensoring the exact sequence $$0\to \O_{S_0}\to \O_{S_1}\oplus \O_{S_2}\to \O_C\to 0$$ by the perfect complexes $R\mathcal{H}om(\I_0,\I_0\otimes p_1^*M_0)$ and $p_1^*M_0$ gives the exact triangles over $S_0\times S_0^{[n]}$:

$$R\mathcal{H}om(\I_0,\I_0\otimes p_1^*M_0)\to\oplus_{i=1}^2  R\mathcal{H}om(\I_i,\I_i\otimes p_1^*M_i)\to  p_1^*M_C,$$
$$ p_1^*M_0\to \oplus_{i=1}^2 p_1^*M_i\to p_1^*M_C,$$ where $\I_i$ is the universal ideal sheaf over $S_i \times S_i^{[-]}$ for $i=0,1,2$, $p_1$ is the projection to the first factor, and all the obvious push forwards by the inclusions of $C$, $S_1$, and $S_2$ into $S_0$ are dropped. The term  $p_1^*M_C$ in the first exact triangle above is because of the relativity condition on ideal sheaves.  

Let $p_2$ be the projection to the second factor of $S_0\times S_0^{[n]}$. Applying $Rp_{2*}$ to the exact triangles above, and taking the difference of the $K$-group classes of the resulting complexes, we get the result.
\end{proof}


We now use some of the notation introduced in Section \ref{Theory}. Consider the $\P^1$-bundle $\P_{C,L}$ defined in \eqref{pcl}. For $M \in \Pic(\P_{C,L})$, define $$m_f=M\cdot f, \quad m_0=M\cdot C_0, \quad l_0=\pi^* L\cdot C_0.$$

\begin{cor}\label{infty}
$Z_{\mathcal{T}}(\P_{C,L}/C_\infty,M)=\Xi(q)^{-\epsilon_{\P_{C,L}}(M)}$ where $$\epsilon_{\P_{C,L}}(M)=m_f^2 l_0+ (1+2m_f)m_0+(1+m_f)(2-2g). $$

\end{cor}
\begin{proof}
Applying Proposition \ref{localization}, we get $$Z_{\mathcal{T}}(\P_{C,L}/C_\infty,M)=Z_{\mathcal{T}}(\P_{C,L,0},M^\infty).$$

The fibers of $M^\infty$ over $C_0$ carry the $\CC^*$-weight $m_ft$, so we can write
\begin{align*}& \delta_{\P_{C,L,0}}(M^\infty) =\\&\int_C \frac{m_f(m_f+1)t^2+\Big((m_f+1)(c_1(C)+c_1(M|_{C_0}))+m_f(-c_1(L)+c_1(M|_{C_0}))\Big)t}{t-c_1(L)}\\
&=m_f^2 L\cdot C_0+ (1+2m_f)M\cdot C_0+(1+m_f)e(C).\end{align*} This finishes the proof.

\end{proof}

 \begin{cor}

Define \begin{align*} R(\P_{C,L},M)=1-\sum_{n>0 }q^{n}\int_{\PR{n}} \frac{c(\mathcal{T}^{[n],\sim}_{S/C}(M))}{1-m_f \Psi_0}.   \end{align*}
Then, 
$$R(\P_{C,L},M)=\Xi(q)^{-\widetilde{\epsilon}_{\P_{C,L}}(M)}$$
where $$\widetilde{\epsilon}_{\P_{C,L}}(M)=m_f^2 l_0+ 2m_f m_0+m_f(2-2g). $$

\end{cor} 

\begin{proof}
By a simple calculation using \eqref{rubber}, $$R(\P_{C,L},M)=Z_{\mathcal{T}}(\P_{C,L},M^0)^\sim,$$ because the fibers of $M^0$ over $C_\infty$ carry the $\CC^*$-weight $-m_ft$. By Carlsson-Okounkov's formula \eqref{eqco} and Remark \ref{quasi}, $$Z(\P_{C,L,0},M^0)=\Xi(q)^{-c_1(M)|_{C_0}-c_1(L)}.$$ Combining these formulas with Corollary \ref{cor:rubber} and the formula for $Z(\P_{C,L,0},M^\infty)$ in the proof of Corollary \ref{infty} , we obtain the result.

\end{proof}

\begin{cor} \label{relative}
Suppose that $D$ is a divisor on $C$ of degree $d$. Then, 
$$Z_{\mathcal{T}}(\P_{C,L}/C_\infty,\pi^*D)=\Xi(q)^{-e(C)-d}.$$
\end{cor}
\begin{proof}
In this case $m_f=\pi^*(D)\cdot f=0$, and $M\cdot C_0=d$. Now use the result of Proposition \ref{localization} and  Carlsson-Okounkov's formula \eqref{eqco}.
\end{proof}

\begin{proof}[Proof of Theorem \ref{mains}]
We prove the theorem for $k=1$ and $C_1=C$, the generalization to higher $k$ is straightforward.  By Carlsson-Okounkov's formula \eqref{eqco} and Corollaries \ref{cor:formul} and 
\ref{relative}, we have 
$$Z_{\mathcal{T}}(S/C,M)=\Xi(q)^{-\eta_{(S,C)}(M)}, \quad \eta_{(S,C)}(M)=e(S)-K_S\cdot M+M^2-e(C)-M\cdot C.$$

The result now easily follows from the fact that 
\begin{align*} 
c_1(T_S[-C])&=c_1(S)-C,  \\c_2(T_S[-C])&=c_2(S)-c_1(S)\cdot C+C^2=c_2(S)-e(C).
\end{align*}

\end{proof}

\begin{remark}
Behrend in \cite{B09} defines the $\mathbb{Q}$-valued function $\chi(X)$ for any DM stack $X$ over $\CC$, which specializes to the Euler characteristic with compact support in the case that $X$ is a scheme. By Gauss-Bonnet theorem \cite[Proposition 1.6]{B09} and Theorem \ref{mains} for a projective surface $S$ we have 
$$1+\sum_{n>0} \chi(\SC{n})q^n=Z_{\mathcal{T}}(S/C,\O).$$ Note that $\chi(\SC{n})$ is different from the alternating sum of the dimensions of the cohomology groups obtained by taking the $\lim_{t\to -1}$ of the formula in \cite[Theorem 1]{S10} for the Poincare polynomial of $\SC{n}$ (see \cite[Section 1.3]{B09}).
\end{remark}

\section{Polynomiality of the top intersection numbers}\label{other}
For any $M \in \Pic(S)$, the rank $n$ tautological bundle is defined by $$M^{[n]}:=p_{2*}p_1^*M\otimes \O_{S\times \SC{n}}/\I,$$ where $p_1,p_2$ are projections to the first and second factors of $S\times \SC{n}$, and $\I$ denotes the universal ideal sheaf. In this section we study two other applications of Theorem \ref{main1}.

For any pair $(S,C)$ of a nonsingular surface and a smooth divisor, and any $n=1,2,\dots$, we define $$\mathcal{L}_{S/C}^{[n]},\; \mathcal{S}_{S/C}^{[n]} :\Pic(S)\to K(\SC{n})$$  $$\mathcal{L}_{S/C}^{[n]}(M)=M^{[n]}\oplus M^{[n]}, \quad \mathcal{S}_{S/C}^{[n]}(M)=-M^{[n]}.$$ As in Lemma \ref{satis}, one can see that the assignments $\mathcal{S}$ and $\mathcal{L}$ satisfy properties (i), (ii), (iii) in Section \ref{sec:main}. For the property (ii), the key point is that by the relativity condition on the ideal sheaves we have $\O_{S\times \SC{n}}/\I \otimes \O_{C\times \SC{n}}=0$.

If $D\subset S$ is an effective divisor, then the $q$-coefficients of the series $$Z_{\mathcal{L}}(S,\O(D))$$ can be interpreted as the self-intersection number of the $n$-dimensional cycle $D^{[n]}\subset \S{n}$, by noting that $e(\mathcal{L}_{S}^{[n]}(\O(D)))=e(\O(D)^{[n]})^2$.

Suppose that the ample divisor $H$ on $S$ is so that the linear system $|H| $ is $3n - 2$-dimensional. Then the $q$-coefficients of the series $Z_{\mathcal{S}}(S,\O(H))$ are the top Segre classes of $\O(H)^{[n]}$, and they can be interpreted as the number of points in $\S{n}$ which do not impose independent conditions on curves in the linear system. They were considered by Donaldson in connection with the computation of the instanton invariants. These top Segre classes were studied in \cite[Section 4.3]{L99} and by other people. They were explicitly computed in \cite{L99} up to $n=7$, and a general formula for their generating function was conjectured, based on these calculations. 

It is known that the $q$-coefficients of the series $$Z_{\mathcal{L}}(S,M),\quad Z_{\mathcal{S}}(S,M)$$ are polynomials in the intrinsic invariants of $S$ such as $M^2$, $M\cdot K_S$, $K_S^2$ and $e(S)$ \cite{EGL99}.  Using the notation of Corollary \ref{cor:formul}, $$Z_{\mathcal{S}}(\mathcal{N} ,M_\P^\infty), \quad Z_{\mathcal{L}}(\mathcal{N} ,M_\P^\infty)$$ can be expressed in terms of the same intersection numbers once they are computed $\CC^*$-equivariantly: $$(M^\infty_\P)^2=0, \quad K^2_\mathcal{N}=C^2+4-4g, \quad K_\mathcal{N}\cdot M^\infty_\P=-M\cdot C, \quad e(\mathcal{N})=e(C)=2-2g,$$ 
For example, using Lehn's calculation,  \begin{align*}\int_{\mathcal{N}^{[3]}}e(-M^{[3]})&=-192M\cdot C+56C^2+72(2-2g).\end{align*}

By Theorem \ref{main1}, we have proven that 
\begin{theorem}\label{thm:intrinsic}
The $q$-coefficients of the generating series $$Z_{\mathcal{L}}(S/C,M) \quad \text{ and } \quad Z_{\mathcal{S}}(S/C,M)$$ are polynomials in $$M^2,\quad  M\cdot K_S,\quad K_S^2, \quad e(S), \quad C^2, \quad M\cdot C, \quad e(C).$$
\end{theorem} \qed

\section{Donaldson-Thomas invariants of $\P^3$} \label{DT}
Let $S_d$ be a nonsingular hypersurface in $\P^3$ of degree $d$. We allow $d=0$ in which case  $S_0$ is empty. Suppose that $M=M(\P^3/S_d,n)$ is the moduli space of stable torsion sheaves on $\P^3$ relative to $S_d$, with Chern characters $$(0,H,-H^2/2,(1/6-n)H^3)\in \oplus_{i=0}^3 H^{2i}(\P^3,\Q),$$ where $H=c_1(\O(1))$.  A general point of the moduli space corresponds to the isomorphism class of a pure 2-dimensional sheaf $\F$ supported on a hyperplane $T \subset \P^3$ which is transversal to $S_d$. By the choice of the Chern character above, $\F$ is isomorphic to the push forward of the ideal sheaf of a length $n$ 0-dimensional subscheme of $T$. 
In an upcoming work, we show that $M$ is complete of virtual dimension $3$ and there is a natural support morphism $$\pi:M\to |H|=\P^3.$$ The Donaldson-Thomas invariant of $M$ can be defined by $$DT(\P^3/S_d,n)\cdot [\P^3]=\pi_*[M]^{vir}.$$ 

We then prove that 
\begin{theorem*}(Gholampour-Sheshmani)
$$DT(\P^3/S_d,n)=\int_{(\P^2/C_d)^{[n]}}e(\T{n}_{\P^2/C_d}(H))$$ where $C_d$ is a generic hyperplane section of $S_d$. The generating series of DT invariants is explicitly given by Theorem \ref{mains} and hence is modular. 

\end{theorem*}

\noindent {\tt{amingh@math.umd.edu}} \\
\noindent {\tt{University of Maryland}} \\
\noindent {\tt{College Park, MD 20742-4015, USA}} \\

\noindent {\tt{sheshmani.1@math.osu.edu}} \\
\noindent {\tt{Ohio State University}} \\
\noindent {\tt{600 Math tower, 231 West 18th avenue, Columbus, Ohio, 43210, USA}} \\

\end{document}